\documentclass[11pt]{article}
\usepackage{amsmath,amsfonts,amssymb,amsthm,graphics,amscd,color,bm}
\usepackage{graphicx}

\usepackage{bibcheck}
\usepackage{cite}
\newtheorem{thm}{Theorem}[section]
\newtheorem{lem}[thm]{Lemma}
\newtheorem{cor}[thm]{Corollary}

\newtheorem{rem}[thm]{Remark}

\theoremstyle{definition}
\newtheorem{definition}[thm]{Definition}

\def\sbmatrix{\left[\begin{array}}
\def\endsbmatrix{\end{array}\right]}

\def\rank{{\mathrm {rank}}}
\def\deg{{\rm deg}}
\def\rdet{{\rm rdet}}
\def\cdet{{\rm cdet}}
\def\Mdet{{\rm Mdet}}
\def\diag{{\rm diag}}
\numberwithin{equation}{section}

\begin{document}

\title{\bf The determinant of the Laplacian matrix of a quaternion unit
gain graph}

\author{Ivan I. Kyrchei$^{a}$\footnote{Corresponding author Email:ivankyrchei26@gmail.com},\qquad Eran Treister $^{b}$, \qquad Volodymyr O. Pelykh$^{a}$,\\
$^a$ Pidstryhach Institute for Applied Problems of Mechanics and Mathematics\\ of NAS of Ukraine, L'viv, Ukraine,\\
$^{b}$ Ben-Gurion University of the Negev, Beer-Sheva,  Israel
}

\date{}

\maketitle

\begin{abstract}
A quaternion unit gain graph is a graph where each orientation of an edge is given a quaternion unit, and the opposite orientation is assigned the inverse of  this quaternion unit. In this paper, we provide a combinatorial description of the determinant of the Laplacian matrix of a quaternion unit gain graph by using row-column noncommutative determinants recently introduced by one of the authors.
A numerical example is presented for illustrating our results.


\end{abstract}

%%%%%%%%%%%%%%%%%%%%%%%%%%%%%%%%

\section{Introduction}

Standardly, we state $\mathbb{C}$ and $\mathbb{R}$, respectively, for the complex and real numbers. An extension of
these fields is the quaternion skew field  the quaternion skew field
$$\mathbb{H}=\{q_0+q_1\mathbf{i}+q_2\mathbf{j}+q_3\mathbf{k}\mid \mathbf{i}^2=\mathbf{j}^2=\mathbf{k}^2=\mathbf{ijk}=-1,\ q_i\in\mathbb{R}, i=0,\ldots,3\},$$
and accordingly, we denote by
 $\mathbb{H}^{m\times n}$ the set of $m\times n$ matrices over $\mathbb{H}$.
Furthermore,
suppose that $q=q_{0}+q_{1}\mathbf{i}+q_{2}\mathbf{j}+q_{3}\mathbf{k}\in \mathbb{H}$, then its conjugate is $\overline{q}=q_{0}-q_{1}\mathbf{i}-q_{2}\mathbf{j}-q_{3}\mathbf{k}$, and its  norm (or modulus) is $|q|=\sqrt{q\overline{q}}=\sqrt{\overline{q}q}=\sqrt{q_{0}^2+q_{1}^2+q_{2}^2+q_{3}^2}$. If $q\neq0$, then the inverse of $q$ is $q^{-1}= \frac{\overline{q}}{|q|^2}$.
Two quaternions $x$ and $y$ are said to be similar if there exists a nonzero
quaternion $u$ such that $u^{-1}xu = y$; this is written as $x \sim y$. By $[x]$ denote the
equivalence class containing $x\in \mathbb{H}$.
For ${\bf A}\in {\mathbb{H}}^{n\times m}$, its  conjugate transpose (Hermitian) matrix is given by  ${\bf A}^{ *}$.
A quaternion matrix ${\bf A}\in{\mathbb{H}}^{n\times n}$ is Hermitian if ${\rm {\bf A}}^{ *}  = {\bf A}$.

Through the paper, by bold capital letters we denote quaternion matrices and  bold lowercase letters mark quaternion vectors and quaternion  units.

Let $\Gamma = (V, E)$ be a simple graph with vertex set $V(\Gamma) = \{v_1, v_2,\ldots, v_n\}$
and edge set $E(\Gamma) = \{e_1, e_2,\ldots, e_m\}$.
A \emph{signed graph} $G= (\Gamma, \sigma)$ consists of an unsigned graph $\Gamma= (V, E)$ and a mapping
$\sigma:E\rightarrow \{\pm 1\}$, the edge labeling. Signed graphs were initially  introduced by Harary \cite{har}, and afterward Zaslavsky extended the matroids of graphs   to matroids of signed graphs in \cite{zas1}. Further development  of the theory of signed graphs was continued in \cite{cam,hou1,hou2,li1,li2}.

Extending a signed graph over complex numbers leads to a complex unit gain graph introduced independently  by Reff \cite{ref1} and Bapat et al \cite{bap1}.

Let $\mathbb{T}=\{z\in\mathbb{C}\mid |z|=1\}$ be the multiplicative group of all complex numbers with absolute value $1$. A complex unit $\mathbb{T}$-gain graph is defined in \cite{ref1} as a graph with the additional structure that each orientation of an edge is given a complex unit, called a \emph{gain}, which is the inverse of the complex unit assigned to the opposite orientation. Note that the definition of a weighted directed graph introduced in  \cite{bap1} is same as a $\mathbb{T}$-gain graph. By definition a $\mathbb{T}$-gain graph is a triple $G = (\Gamma, \mathbb{T}, \varphi)$
consisting of an underlying graph $\Gamma = (V, E)$, the circle group $\mathbb{T}$ and a function $\varphi: \overrightarrow{E}(\Gamma) \rightarrow \mathbb{T}$ (called the \emph{gain function}), such
that $\varphi(e_{ij}) = \varphi(e_{ji})^{-1}$. In \cite{ref1}, it was defined   its associated matrices and eigenvalue bounds for the adjacency and Laplacian matrices were obtained.  More properties of $\mathbb{T}$-gain graphs can be found in  \cite{ref2,ham,zas2,zas3,zas4}. Especially,  our attention attracted the paper \cite{y_wang1}, where Wang et al provide a combinatorial description  for the determinant of the Laplacian matrix of a $\mathbb{T}$-gain graph.

Recently, a few researches started to extend the sufficiently developed theory of the complex unit $\mathbb{T}$-gain graphs to gain graphs where the gains can be any quaternion units.
 Belardo et al.  \cite{bel1} defined the adjacency, Laplacian and incidence matrices for a quaternion unit gain graph and studied their properties which generalize several fundamental results from spectral graph theory of ordinary graphs, signed graphs and complex unit gain graphs.
The paper \cite{shou_q} is devoted to the study of the row left rank of a quaternion unit gain
graph.

Let $U(\mathbb{H})=\{q\in\mathbb{H}\mid |q|=1\}$ be the circle group which is the multiplicative group of all quaternions with absolute value 1.
Suppose that $\Gamma = (V, E)$ is the simple graph with the set of vertices  $V = \{v_1, v_2,\ldots, v_n\}$ and edges in $E$ denoted by $e_{ij} = v_iv_j$.
 Moreover, $\overrightarrow{E} =\overrightarrow{E} (\Gamma)$ is defined to be the set of oriented edges of the gain graph. By $e_{ij}$ we denote  the oriented edge from $v_i$ to $v_j$ and the gain of $e_{ij}$ is denoted by $\varphi(e_{ij})$. Even though this is the same notation for an oriented edge from $v_i$ to $v_j$ it will always be clear whether an edge or an oriented edge is being used.

 Hence, an $U(\mathbb{H})$-\emph{gain graph} is a triple $G = (\Gamma, U(\mathbb{H}),\varphi)$
consisting of an underlying graph $\Gamma = (V, E)$, the \emph{circle group} $U(\mathbb{H})$ and the \emph{gain function} $\varphi : \overrightarrow{E}(\Gamma)\rightarrow U(\mathbb{H})$, such
that $\varphi(e_{ij}) =\varphi(e_{ji})^{-1}=\overline{\varphi(e_{ji})}$.

Similarly to a $\mathbb{T}$-gain graph, a $U(\mathbb{H})$-gain graph $G$ has standard  matrix representations such as an incidence matrix, an adjacency matrix and a Laplacian matrix.
 Taking into account the non-commutativity of quaternions, the tasks dealing with matrix representations of  quaternion unit gain graphs are  more complicated than with the complex ones.
Difficulties arise even in defining the determinant of a quaternion matrix as a determinant of a quadratic matrix with noncommutative entries (as called a noncommutative determinant)   \cite{as,coh,zhan}.
In \cite{bel1}, it is used the Moore noncommutative determinant,  which is defined only for a Hermitian matrix, narrowing the field of research and applications.

The main goal of this paper is to give a combinatorial description for the Laplacian matrix of an arbitrary $U(\mathbb{H})$-gain graph.
Especially, to explore matrix representations of a $U(\mathbb{H})$-gain graph we  rely on the theory of column-row determinants recently developed in \cite{kyr2,kyr_laa,kyr3}, and use the paper \cite{y_wang1}  as a research scheme for our paper.
In  \cite{y_wang1}, Wang et al considered the same task for a complex unit gain graph by giving a series of statements expressed by lemmas and the final resulting theorem.   Although we give  statements of some lemmas  that are similar to ones in  \cite{y_wang1}, but their proofs are substantially different due to features of quaternion matrices.

The remainder of our article is directed as follows.
Some preliminaries of quaternion matrices, especially properties of column-row quaternion determinants that are needed to obtain the main results are given in Section  \ref{sec:prelim}.
The main results related to a combinatorial description the determinant of the Laplacian matrix of an arbitrary $U(\mathbb{H})$-gain graph are derived in Section \ref{sec:result}.
A numerical example of our results is given in Section \ref{sec:examp}. Finally, in Section \ref{sec:conc}, the conclusions are drawn.

\section{Preliminaries. Features of quaternion matrices}\label{sec:prelim}

Due to noncommutativity of quaternions, there is the problem of defining  a determinant of a  matrix with noncommutative entries  (that is also called as the  noncommutative determinant). One of the ways is a transformation of a quaternion matrix into corresponding real or complex matrices, and after that to use the usual determinant \cite{di,zhan}.  But such determinants take on real or complex values only, and their functional properties are also  restricted in comparing with the usual determinant.

 Another way is to define  a noncommutative determinant
  in usual way as the alternating sum of $n!$ products of entries
    but by specifying a certain order of coefficients in each term.  Moore  was the first who achieved the fulfillment success in a construction of such determinant \cite{mo,dy}.  However, this construction was defined for Hermitian matrices only, and not for all square
matrices. The recently developed theory of row-column determinants in  \cite{kyr2,kyr3} provides a
solution to the problem of extending Moore's determinant to all quadratic quaternion matrices.

Below we define, for ${\bf A}=(a_{ij}) \in {\mathbb{H}}^{n\times n}$, a method to produce $n$ row ($\mathfrak{R}$-)determinants and $n$ column ($\mathfrak{C}$-)determinants by stating a certain order of factors in each term.

\begin{definition}\cite{kyr2}\label{def:rdet} \emph{The $i$th $\mathfrak{R}$-determinant} of ${\rm {\bf A}}$, for an arbitrary row index $i \in I_{n}=\{1,\ldots,n\}$, is given by
 \begin{align*}{\rm{rdet}}_{ i} {\rm {\bf A}} :=&
\sum\limits_{\sigma \in S_{n}} \left( { - 1} \right)^{n - r}({a_{i{\kern 1pt} i_{k_{1}}} } {a_{i_{k_{1}}   i_{k_{1} + 1}}} \ldots   {a_{i_{k_{1}+ l_{1}}
 i}})  \ldots  ({a_{i_{k_{r}}  i_{k_{r} + 1}}}  \ldots  {a_{i_{k_{r} + l_{r}}  i_{k_{r}} }}),
\end{align*}
whereat $S_{n}$ denotes the symmetric group on $I_{n}$, while the permutation $\sigma$ is defined as a product of mutually disjunct subsets ordered from the left to right by the rules
\begin{align*}
\sigma = \left(
{i\,i_{k_{1}}  i_{k_{1} + 1} \ldots i_{k_{1} + l_{1}} } \right)\left(
{i_{k_{2}}  i_{k_{2} + 1} \ldots i_{k_{2} + l_{2}} } \right)\ldots \left(
{i_{k_{r}}  i_{k_{r} + 1} \ldots i_{k_{r} + l_{r}} } \right),\\
i_{k_{t}}  <
i_{k_{t} + s},~~i_{k_{2}} < i_{k_{3}}  < \cdots < i_{k_{r}},~~\forall ~ t = 2,\ldots,r,~~~s =1,\ldots,l_{r}.
\end{align*}
\end{definition}
\begin{definition}\cite{kyr2}\label{def:cdet}For an arbitrary column index $j \in I_{n}$, \emph{the $j$th $\mathfrak{C}$-determinant}  of ${\rm {\bf A}}$ is defined as the sum
 \begin{align*}{\rm{cdet}} _{{j}}  {\bf A} =&
\sum\limits_{\tau \in S_{n}} ( - 1)^{n - r}(a_{j_{k_{r}}
j_{k_{r} + l_{r}} } \cdots a_{j_{k_{r} + 1} j_{k_{r}} })  \cdots  (a_{j j_{k_{1} + l_{1}} }  \cdots  a_{ j_{k_{1} + 1} j_{k_{1}} }a_{j_{k_{1}}  j}),
\end{align*}
in which a permutation $\tau$ is ordered from the right to left in the following way:
\begin{align*}
\tau =
\left( {j_{k_{r} + l_{r}}  \cdots j_{k_{r} + 1} j_{k_{r}} } \right)\cdots
\left( {j_{k_{2} + l_{2}}  \cdots j_{k_{2} + 1} j_{k_{2}} } \right){\kern
1pt} \left( {j_{k_{1} + l_{1}}  \cdots j_{k_{1} + 1} j_{k_{1} } j}
\right),\\j_{k_{t}}  < j_{k_{t} + s}, ~~~j_{k_{2}}  < j_{k_{3}}  < \cdots <
j_{k_{r}},~~\forall ~ t = 2,\ldots,r,~~~s =1,\ldots,l_{r}.
 \end{align*}
\end{definition}
In general, row  and column determinants are not equal to each other. But for an Hermitian matrix  ${\rm {\bf A}}$, we have, following \cite{kyr3},
 ${\rdet}_{1}  {\bf A} = \cdots = {\rdet}_{n} {\bf
A} = {\cdet}_{1}  {\bf A} = \cdots = {\cdet}_{n}  {\bf
A} \in  {\mathbb{R}}.$
From this,  \emph{the determinant of a Hermitian matrix} can be defined unambiguously by setting
\begin{align}\label{eq:det_her}\det {\bf A}: = \rdet_i
{\bf A} = \cdet_i {\bf A}=\alpha\in \mathbb{R}\end{align}
 for all $i =1,\ldots,n$.
  This determinant  is the same as the determinant of a Hermitian matrix defined by Moore \cite{mo}, $\Mdet {\bf A}=\alpha$, for which an order of disjoint  circles  does not matter.

  Properties of the row-column determinants
 have been completely explored by  in \cite{kyr3}.
Below we give give some properties of  row-column determinants and results  obtained in \cite{kyr3} that will be used throughout this paper.
\begin{lem}\cite{kyr3}\label{lem:det_cong} If ${\bf A} \in {\mathbb{H}}^{n\times n}$, then $\rdet_i
{\bf A}^* = \overline{\cdet_i {\bf A}}$ for all $i =1,\ldots,n$.
\end{lem}
Let ${\bf
a}_{i.}$ be the $i$th row and ${\bf a}_{.j}$ be the $j$th column of  ${\bf A} \in {\mathbb{H}}^{n\times n}$. Denote by ${\bf A}_{i.} (\bf b)$
 (${\bf A}_{.j} (\bf c)$) a matrix obtained from ${\bf A}$ by replacing its $i$th row ($j$th column) with the row vector $\bf b$ (the column vector  $\bf c$).
\begin{lem}\cite{kyr3}\label{lem:row_comb} If the $i$th row of a Hermitian matrix ${\bf A} \in {\mathbb{H}}^{n\times n}$   is added a left
linear combination of its other rows, then \begin{align*}
   {\rdet}_{i}  {\bf A}_{i . } \left({\bf
a}_{i.} + c_{1} \cdot {\bf a}_{i_{1} .} + \cdots + c_{k}
\cdot  {\bf a}_{i_{k} .} \right) =
 {\cdet} _{i}{\bf A}_{i .} \left( {\bf a}_{i.} + c_{1} \cdot
{\bf a}_{i_{1} .} + \cdots + c_{k} \cdot {\bf
a}_{i_{k} .}  \right) = \det {\bf A},
\end{align*}
 where $ c_{l} \in {\mathbb{H}}$ for all $l = {1,\ldots,k}$ and
$\{i,i_{l}\}\subset I_{n}$.
\end{lem}
\begin{lem}\cite{kyr3}\label{lem:col_comb} If the $j$th column of a Hermitian matrix ${\bf A} \in {\mathbb{H}}^{n\times n}$   is added a right
linear combination of its other columns, then
\begin{align*}
   {\cdet} _{j} {\bf A}_{.j} \left(  {\bf a}_{.j} +
{\bf a}_{.j_{1}}
 c_{1} + \cdots + {\bf a}_{.j_{k}}   c_{k}
\right)
  ={\rdet}_{j} {\bf A}_{.j} \left( {\bf a}_{.j} + {\bf a}_{.j_{1}}   c_{1} + \cdots + {\bf a}_{.j_{k}}  c_{k}  \right) = \det {\bf A},
\end{align*}
where $c_{l} \in{\mathbb{H}}$ for all $l =
{1,\ldots,k}$ and $\{j,j_{l}\}\subset J_{n}$.
\end{lem}
The following criterion of invertibility of an arbitrary quadratic quaternion  matrix holds.
\begin{lem}\cite{kyr3}\label{lem:crit_nonsing} Let ${\bf A} \in {\mathbb{H}}^{n\times n}$. Then the following
statements are equivalent.

(i) ${\bf A}$ is invertibility.

(ii) $\det {\bf A}{\bf A}^*=\det {\bf A}^*{\bf A}\neq 0$.

(iii) The rows of ${\bf A}$ are left-linearly independent.

(iv) The columns of ${\bf A}$ are right-linearly independent.
\end{lem}
From Lemmas \ref{lem:row_comb} and \ref{lem:col_comb} it is evidently follows that row vectors the span left linear quaternion vector space $\mathcal{H}_l$ and  column vectors form the right linear quaternion vector space $\mathcal{H}_r$.

For ${\bf A} \in {\mathbb{H}}^{n\times m}$,  the (left) \emph{row rank} is defined to be the maximum number of its left-linearly independent rows and the (right) \emph{column rank} is the maximum number of its right-linearly independent columns. The \emph{determinantal rank} of ${\bf A}$ can be defined as the largest possible size of a nonzero principal minor of its corresponding Hermitian matrices ${\bf A}{\bf A}^*$ or ${\bf A}^*{\bf A}$. All these ranks are equivalent to each other and the next holds.

\begin{lem}\cite{kyr3}\label{lem:rank} If ${\bf A}\in{\mathbb{H}}^{m\times n}$, then
 $\rank \,{\bf A}=\rank\, {\bf A}^*{\bf A}=\rank\, {\bf A}{\bf A}^*$.
\end{lem}

As well-known, for ${\bf A} \in {\mathbb{H}}^{n\times n}$ and $\lambda \in {\mathbb{H}}$, its \emph{left} and \emph{right eigenvalues} are introduced by the equations ${\bf A}\mathbf{x}=\lambda \mathbf{x}$ and  ${\bf A}\mathbf{x}= \mathbf{x}\lambda$, respectively. Especially, a right  eigenvalue  seems more natural in eigenpair  with   its associated right (column) eigenvector.
The following results regarding quaternion eigenvalues are  known.
\begin{lem}\cite{zhan}\label{lem:zh_eigen}
Suppose that ${\bf A} \in {\mathbb{H}}^{n\times n}$ has right eigenvalues $h_1+k_1\mathbf{i}$,..., $h_n+k_n\mathbf{i}$, where $h_i, k_i \in {\mathbb{R}}$ and $ k_i \geq 0$ for all $i=1,\ldots,n$. Then the spectra of right eigenvalues of  ${\bf A}$ is
$$\sigma_r({\bf A})=[h_1+k_1\mathbf{i}]\cup\cdots\cup[h_n+k_n\mathbf{i}].$$
\end{lem}
\begin{lem}\cite{kyr4}The matrix ${\bf A} \in {\mathbb{H}}^{n\times n}$ is Hermitian if and only if there are a unitary matrix $\bf U$ and a real diagonal matrix ${\bf D}={\rm diag} (\lambda_1,\ldots,\lambda_n)$, where $\lambda_i\in \mathbb{R}$ is a right eigenvalue of ${\bf A}$ for all $i=1,\ldots,n$, such that
${\bf A}={\bf U}{\bf DU}^*$ and $\det {\bf A}=\lambda_1\cdot\cdots\cdot\lambda_n.$
\end{lem}
\begin{lem}\cite{kyr4}\label{lem:semidef}Let ${\bf A}\in{\mathbb{H}}^{n\times m}$ and $\rank {\bf A}=r$.
Then ${\bf A}^*{\bf A}$ and ${\bf A}{\bf A}^*$ are both positive
semi-definite matrices, and $r$ real nonzero eigenvalues of  ${\bf A}^*{\bf A}$ and ${\bf A}{\bf A}^*$ coincide.
\end{lem}
\begin{definition}Let ${\bf A} \in {\mathbb{H}}^{n\times n}$ be Hermitian and $t\in \mathbb{R}$ be a real variable. The polynomial $p_A (t) =\det (\mathbf{I}t-\mathbf{A})$ is said to be the \emph{characteristic polynomial} of ${\bf A}$.
\end{definition}
The following properties are the extension  of the characteristic polynomial of a complex matrix to a quaternion Hermitian matrix.
\begin{lem}\cite{kyr3}\label{lem:char_pol} If ${\bf A}\in{\mathbb{H}}^{n\times n}$ is Hermitian, then
$$p_A(t)=t^n-d_1t^{n-1}+d_2t^{n-2}-\cdots+(-1)^nd_n,$$
where $d_s=\sum_{\alpha\in I_{s,n}} \det ({\bf A})_\alpha^\alpha$ is the sum of principle minors of  ${\bf A}$ and $d_n=\det \mathbf{A}$. Here $({\bf A})_\alpha^\alpha $ denotes   a principal submatrix of ${\bf A}$
   whose rows and columns are indexed by $\alpha : = \left\{
{\alpha _{1},\ldots,\alpha _{s}} \right\} \subseteq {\left\{
{1,\ldots ,n} \right\}}$ and
 $I_{s,n}: = {\left\{ {\alpha :\,\, 1 \le \alpha_{1} < \cdots< \alpha_{s} \le n} \right\}}
$ for all $s=1,\ldots, n-1$.
\end{lem}
From Lemmas \ref{lem:semidef} and \ref{lem:char_pol} the next lemma evidently follows.
\begin{lem}\label{lem:eqdet}Suppose that ${\bf A}\in{\mathbb{H}}^{n\times m}$ and $\rank {\bf A}=r$. Then for any $s\leq r$, we have
  $$\sum_{\alpha\in I_{s,m}} \det ({\bf A}^*{\bf A})_\alpha^\alpha=\sum_{\alpha\in I_{s,n}} \det ({\bf A}{\bf A}^* )_\beta^\beta
   $$
   where
    $\alpha : = \left\{
{\alpha _{1},\ldots,\alpha _{s}} \right\} \subseteq {\left\{
{1,\ldots ,m} \right\}}$, and
 $I_{s,
m}: = {\left\{ {\alpha :\,\, 1 \le \alpha_{1} < \cdots< \alpha_{s} \le m} \right\}}
$; similarly,  $\beta : = \left\{
{\beta _{1},\ldots,\beta _{s}} \right\} \subseteq {\left\{
{1,\ldots ,n} \right\}}$, and
 $I_{s,
n}: = {\left\{ {\beta :\,\, 1 \le \beta_{1} < \cdots< \beta_{s} \le n} \right\}}.
$

\end{lem}

\section{The determinant of a quaternion unit gain graph}\label{sec:result}

First we introduce matrices that are related and are used to  represent the $U(\mathbb{H})$-gain graph.
\begin{definition}\label{def:incid_matr}Let $G = (\Gamma, \varphi)$ be a $U(\mathbb{H})$-gain graph with vertex set $\Gamma(V) = \{v_1, v_2,\ldots,v_n\}$ and edge set
 $\Gamma(\overrightarrow{E}) = \{e_1, e_2,\ldots,e_m\}$.
 The (vertex-edge) \emph{incidence matrix} $\mathbf{H}(G) = (\eta_{ve})$ is any $n\times m$ matrix with entries in $U(\mathbb{H})\bigcup\{0\}$, where each column corresponds to an edge $e_k =e_{ij}=\overrightarrow{v_iv_j}\in \overrightarrow{E}$ for all $k=1,\ldots,m$, and has all zero entries except two nonzero entries $\eta_{v_je_k}=1$ and $\eta_{v_ie_k}=-\varphi(e_{ij})$, i.e.
 \begin{align*}
 \eta_{ve}=\begin{cases}1,& \mbox{if}\, v=v_j\, \mbox{and}\,  e =e_{ij}\in \overrightarrow{E},\\
 -\varphi(e_{ij}),&\mbox{if}\, v=v_i\, \mbox{and}\,  e =e_{ij}\in \overrightarrow{E},\\
 0,&\mbox{otherwise.}
 \end{cases}
 \end{align*}
\end{definition}
Definition \ref{def:incid_matr} is a particular case of an incidence matrix related to a  $U(\mathbb{H})$-gain graph defined by   Belardo et al \cite{bel1}.
\begin{definition}\label{def:adj_matr}Let $G = (\Gamma, \varphi)$ be a $U(\mathbb{H})$-gain graph with vertex set $\Gamma(V) = \{v_1, v_2,\ldots,v_n\}$.
 The (edge-edge)  \emph{adjacency matrix} $\mathbf{A}(G) = (a_{ij})\in \mathbb{H}^{n\times n}$ is defined by
 \begin{align*}
 a_{ij}=\begin{cases}
 \varphi(e_{ij}),&\mbox{if}\, v_i\, \mbox{is adjacent to}\,  v_j,\\
 0,&\mbox{otherwise.}
 \end{cases}
 \end{align*}
\end{definition}
If $v_i$ is adjacent to $v_j$, then $a_{ij}=\varphi(e_{ij}) =\varphi(e_{ji})^{-1}=\frac{\overline{\varphi(e_{ji})}}{|\varphi(e_{ji})|^2}=\overline{\varphi(e_{ji})}=\overline{a_{ji}}$ for all $i,j=1,\ldots,n$. Therefore, the matrix $\mathbf{A}(G)$ is Hermitian.

 The number of edges attached to each vertex is called the \emph{degree} of the vertex, and it is denoted by $\deg(v_j)$ for each vertex $v_j$ for all $j=1,\ldots,n$.
\begin{definition}\label{def:lap_matr}
Let $G = (\Gamma, \varphi)$ be a $U(\mathbb{H})$-gain graph.
The \emph{Laplacian matrix} (Kirchhoff matrix or admittance matrix) is defined as $\mathbf{L}(G) =\mathbf{D}(\Gamma)-\mathbf{A}(G)$, where $\mathbf{D}(\Gamma)=\diag \left(\deg(v_1),\ldots,\deg(v_n)\right)$ is the diagonal matrix of the degrees of vertices of $\Gamma$.
 \end{definition}
Note that by Definition \ref{def:lap_matr}, $\mathbf{L}(G)$  coincides with the Laplacian matrix of the underlying graph of $\Gamma$ if $G$ has gain 1, with the signless
Laplacian matrix of $\Gamma$ if $G$ has gain -1, and with the Laplacian matrix of a signed graph  if G has gains $\pm 1$.

 It is evident that  $\mathbf{L}(G)$ is also Hermitian.
From \cite[Lemma 3.1]{bel1}, $\mathbf{L}(G) = \mathbf{H}(G)\mathbf{H}(G)^*$.
By Lemma \ref{lem:semidef},  $\mathbf{L}(G)$ is a positive semi-definite matrix, and $\det \mathbf{L}(G)\geq 0$.

Let the gain of a walk $W=v_1e_{12}v_2e_{23}v_3\ldots v_{k-1}e_{k-1,k}v_k$ be
$$\varphi(W) =\varphi(e_{12})\varphi(e_{23})\ldots\varphi(e_{k-1,k}).$$
A walk $W$ is \emph{neutral} if $\varphi(W)=1$. A walk such that $v_k=v_1$, where $k\geq 3$, will be called a \emph{cycle}. An edge set $S\subseteq \Gamma$ is \emph{balanced} if every cycle $C\subseteq S$ is neutral. A subgraph is balanced if its edge set is balanced.

A connected $U(\mathbb{H})$-gain graph containing no cycles is called a $U(\mathbb{H})$--\emph{gain tree}. Since a $U(\mathbb{H})$-gain tree of order $n$ contains exactly $n-1$ edges, then $ \mathbf{H}(G)\in \mathbb{H}^{n,n-1}$ and by Lemma \ref{lem:rank}, $\rank \, \mathbf{H}(G)=\rank\, \mathbf{L}(G)< n$. From this  the next lemmas follow.

\begin{lem}\label{lem:det_T} Let $T$ be an arbitrary $U(\mathbb{H})$--gain tree with Laplacian matrix $\mathbf{L}(T)$. Then
$$\det\, \mathbf{L}(T)=0.$$
\end{lem}
Let  $C = v_1e_{12}v_2\ldots v_{s-1}e_{s-1,s}v_s(= v_1)$ be a cycle with $s\geq3$ edges, where $v_j$ adjacent to $v_{j+1}$ for $j = 1, 2,\ldots,s-1$ and
$v_1$ incident to $v_s$. The gain of $C$ is defined by
$$\varphi(C) =\varphi(e_{12})\varphi(e_{23})\ldots\varphi(e_{s-1,s})\varphi(e_{s1}).$$

By Definition \ref{def:lap_matr} of the Laplacian matrix of a $U(\mathbb{H})$-gain graph $G$,  $\mathbf{L}(G)=(l_{ij})$ with $l_{ij}=-\varphi(e_{i,j})$ when $v_i$ adjacent to $v_j$. Hence, the gain of $C$ in terms of the entries of its Laplacian matrix can be defined as follows,

 \begin{align}\label{rep_C}\varphi(C) =(-1)^sl_{12}l_{23}\ldots l_{s-1,s}l_{s,1}.\end{align}
\begin{lem}\label{lem:det_C} Let $C$ be a $U(\mathbb{H})$-gain cycle on $n\geq3$ edges  with its incidence and Laplacian matrices,  $\mathbf{H}(C)$ and $\mathbf{L}(C)$, respectively. Then
$$\rdet_1 \mathbf{H}(C)=(1-\varphi(C)),\,\,\det \mathbf{L}(C)=(1-\varphi(C))\overline{(1-\varphi(C))}.$$
\end{lem}
\begin{proof}Let $\mathbf{H}(C)$ be the vertex-edge incident matrix of $C$ whose rows  correspond  to the vertices ${v_1, v_2,\ldots,v_n}$ and columns to the edges ${e_1,e_2,\ldots,e_n}$. Without loss of generality, suppose that $e_1$ incident to $v_1$ and $v_n$, and other vertices $v_j$ and $v_{j+1}$ are two ends of the edge $e_{j+1}$ for $j = 1, 2,\ldots, n-1$. Hence, the nonzero entries of $\mathbf{H}(C)=(\eta_{ij})$ are $\eta_{jj}=1$ for all $j = 1,\ldots,n$,   $\eta_{j,j+1}=-\varphi(e_{j,j+1})$ for all $j = 1,2,\ldots,n-1$, and $\eta_{n,1}=-\varphi(e_{n,1})$. Following Definition \ref{def:rdet},
$$
\rdet_1\mathbf{H}(C)=\prod_{k=1}^{n}\eta_{kk}+(-1)^{n-1}\left((-1)^{n}\eta_{12}\eta_{23}\ldots\eta_{n1}\right)=1-\varphi(C).
$$
For the matrix $\mathbf{H}(C)^*=(\eta^*_{ij})$, we have  $\eta^*_{jj}=1$ for all $j = 1,\ldots,n$, $\eta^*_{j+1,j}=-\overline{\varphi(e_{j,j+1}})$ for all $j = 1,2,\ldots,n-1$, and $\eta^*_{1,n}=-\overline{\varphi(e_{n,1})}$ From Definition \ref{def:cdet} it follows that
$$
\cdet_1\mathbf{H}(C)^*=\prod_{k=1}^{n}\eta^*_{kk}+(-1)^{n-1}\left((-1)^{n}\eta^*_{1n}\ldots\eta^*_{32}\eta^*_{21}\right)
=\overline{1-\varphi(C)}.
$$
Moreover, by Lemma \ref{lem:det_cong}, $\cdet_1\mathbf{H}(C)^*=\overline{\rdet_1\mathbf{H}(C)}$.

Now, we pay attention to the Laplacian matrix of a $U(\mathbb{H})$-gain graph $G$.
Taking the
structure of the matrix $\mathbf{H}(C)$ into account, the nonzero entries of  $\mathbf{L}(G)=(l_{ij})$ are $l_{jj}=2$ for all $j = 1,\ldots,n$,   $l_{j,j+1}=-\varphi(e_{j,j+1})$ and $l_{j+1,j}=-\overline{\varphi(e_{j,j+1}})$ for all $j = 1,2,\ldots,n-1$,  $l_{n,1}=-{\varphi(e_{n,1})}$ and $l_{1,n}=-\overline{\varphi(e_{n,1})}$. Notice that for the cycles of a second order, we have   $l_{j,j+1}l_{j+1,j}=\varphi(e_{j,j+1})\overline{\varphi(e_{j,j+1}})=1$ and $l_{j+1,j}l_{j,j+1}=\overline{\varphi(e_{j,j+1}})\varphi(e_{j,j+1})=1$.

By \eqref{eq:det_her}, we put $\det \mathbf{L}(G)=\rdet_1 \mathbf{L}(G)$ and will be calculate it by Definition \ref{def:rdet}. In accordance to  a number $k$ of cycles of a second order in each term of $\rdet_1 \mathbf{L}(G)$, we have the following sets  of terms and their sums in $\rdet_1 \mathbf{L}(G)$.

 \begin{align*}k=0,&~L_1=l_{11}l_{22}\ldots l_{nn}=2^n,~~~~~~~~~~~~~~~~~~~~~~~~~~~~~~~~~~~~~~~~~~~~~~~~~~~~~~~\allowdisplaybreaks\\
 k=1,&~L_2=(-1)^{n-(n-1)}(l_{12}l_{21}l_{33}\ldots l_{nn}+l_{1n}l_{n1}l_{22}\ldots l_{n-1,n-1}+\\~~~~~~~~~~~~~~~~~~~~~~&~~~~+\sum_{j=2}^{n-1}l_{11}\ldots l_{j,j+1}l_{j+1,j}\ldots l_{nn})=(-1)^1{{n}\choose{1}}2^{n-2}=-n2^{n-2},\allowdisplaybreaks\\
  k=2,&~L_3=(-1)^{n-(n-2)}(\sum_{m_1}l_{12}l_{21}\ldots l_{m_1,m_1+1}l_{m_1+1,m_1}\ldots l_{nn}+\\~~~~~~~~~~~~~~~~~~~~~~&~~~~+\sum_{m_2}l_{1n}l_{n1}\ldots l_{m_2,m_2+1}l_{m_2+1,m_2}\ldots l_{n-1,n-1}+\\~~~~~~~~~~~~~~~~~~~~~~&~~~~+\sum_{i,j}l_{11}\ldots l_{i,i+1}l_{i+1,i}\ldots l_{j,j+1}l_{j+1,j}\ldots l_{nn})=\\~~~~~~~~~~~~~~~~~~~~~~&~~~~=(-1)^2\left[2(n-3)+{{n-3}\choose{2}}\right]2^{n-4}=\frac{n-3}{2}n2^{n-4}.\\
 \end{align*}
By similarly continuing, we obtain
\begin{align*}
 &~L_{k+1}=(-1)^k\left[2{{n-k-1}\choose{k-1}}+{{n-k-1}\choose{k}}\right]2^{n-2k},~~~{\rm ~for~ any}~~~k\leq\left[\frac{n}{2}\right].
 \end{align*}
Using Pascal's rule for the binomial coefficients, it can be express as follows
 \begin{align*}
  &~L_{k+1}=(-1)^k\frac{(n-k-1)!}{k!(n-2k)!}n2^{n-2k},~~~{\rm ~for~ any}~~~k\leq\left[\frac{n}{2}\right].
 \end{align*}

The sum of the last terms with a maximal number of cycles of a second order are
 \begin{align*}
  k=m,&~L_{m+1}=(-1)^m2~~~{\rm ~when}~ n=2m~{\rm ~is~even},\\
   k=\left[\frac{n}{2}\right]=m,&~L_{m+1}=(-1)^m2(2m+1)~~~{\rm ~when}~ n=2m+1~{\rm ~is~not~~even}
 \end{align*}
 Finally, taking into account \eqref{rep_C}  we  represent two terms with the cycles $\varphi(C)$ and its   conjugate,
 \begin{align*}
 L_{m+2}=&(-1)^{n-1}l_{12}l_{23}l_{34}\ldots l_{n1}=-\varphi(C),\\
L_{m+3}=&(-1)^{n-1}l_{1n}l_{n,n-1}\ldots l_{32} l_{21}=-\overline{\varphi(C)}.
 \end{align*}
Hence,
\begin{align*}
\det \mathbf{L}(G)=\rdet_1 \mathbf{L}(G)=\sum_{k=0}^m {(-1)^k\frac{(n-k-1)!}{k!(n-2k)!}n2^{n-2k}}-\varphi(C)-\overline{\varphi(C)},
\end{align*}
where $m=\left[\frac{n}{2}\right]$ is the integer part of $n$.
Since
\begin{align*}
\sum_{k=0}^m {(-1)^k\frac{(n-k-1)!}{k!(n-2k)!}n2^{n-2k}}=2,
\end{align*}
and
$$\varphi(C)\overline{\varphi(C)}=(l_{12}l_{23}l_{34}\ldots l_{n1})(l_{1n}l_{n,n-1}\ldots l_{32} l_{21})=1,$$
then
\begin{align*}
\det \mathbf{L}(G)=2-\varphi(C)-\overline{\varphi(C)}=(1-\varphi(C))\overline{(1-\varphi(C))}.
\end{align*}

\end{proof}

\begin{rem}Even though $\mathbf{L}(G) = \mathbf{H}(G)\mathbf{H}(G)^*$, but the (Hermitian) determinant  $\det \mathbf{L}(G)$ is not a multiplicative map regarding to matrices $\mathbf{H}(G)$ and $\mathbf{H}(G)^*$, in general. An exception can be in the case when $G$ has a unique cycle $C$. In \cite[Lemma 6.7]{bel1}, it is proven that $\Mdet \mathbf{L}(G)=\Mdet (\mathbf{H}(G)) \Mdet(\mathbf{H}(G)^*)$  holds in this case  only under the hypothesis that all edges but one are neutral.
\end{rem}
From Lemma \ref{lem:det_C} evidently follows the next.
\begin{cor} Let $C$ be a $U(\mathbb{H})$-gain cycle on $n\geq3$ edges and $\mathbf{L}(C)$ be its Laplacian matrix. Then $det(L(C)) = 0$ if and only if $C$ is
balanced.
\end{cor}
Similar to \cite{y_wang1}, we call the cycle $C$ \emph{real unbalanced} if $\varphi(C) = -1$, and \emph{imaginary unbalanced} if $\varphi(C) =\pm i_s$, where $i_s\in \{\mathbf{i},\mathbf{j},\mathbf{k}\}$. It's evident that
\begin{align*}
\det \mathbf{L}(C)=&4,~~  {\rm if}~ C~ {\rm is~ real~ unbalanced},\\
\det \mathbf{L}(C)=&2,~~  {\rm if}~ C~ {\rm is~imaginary~ unbalanced}.\\
\end{align*}
A connected graph containing exactly one cycle is called a  \emph{unicyclic graph}.

\begin{lem}\label{lem:det_unic} Let $G$ be a unicyclic  $U(\mathbb{H})$-gain graph with the unique cycle  $C$. Then
\begin{align}\label{eq:det_unc}\det \mathbf{L}(G)=\det \mathbf{L}(C).\end{align}
\end{lem}
\begin{proof}If a unicyclic $U(\mathbb{H})$-gain graph $G$ does not contain no pendant vertices, then all vertices belong to a cycle, and Eq. \eqref{eq:det_unc} is evident. Suppose that  $G$ contains a pendant vertex $v$.
Without loss of generality, let this vertex  $v_1$ and its unique neighbor vertex  correspond the
first two rows and columns of $L(G)$ such that $l_{11}=1$ and $l_{12}=\overline{l_{21}}$ are corresponding gains on  $e_{12}=v_1v_2$ and  $e_{21}$. By left multiplying the first row by $-l_{21}$ and adding it to the second row,  we obtain a new  matrix  $\mathbf{L}'(G)$ with $l'_{21}=0$ and $l'_{22}=l_{22}-l_{12}l_{21}=l_{22}-1$. The principal submatrix of $\mathbf{L}'(G)$ by deleting the first row and the first column equals the Laplacian matrix $\mathbf{L}(G-v_1)$ of the graph $G-v_1$ obtained from $G$ by  separation the edge $e_{12}=v_1v_2$. It's evident that $\mathbf{L}(G-v_1)$ is Hermitian. Since, $l_{11}=1$, then by Lemma \ref{lem:col_comb},
$$
\det \mathbf{L}(G)=\rdet_2 \mathbf{L}(G)_{2.}(\mathbf{l}_{2.}-l_{12}\mathbf{l}_{1.})=\det \mathbf{L}(G-v_1).
$$
Further, if the  vertex $v_2$  turns out as  pendant in $\mathbf{L}(G-v_1)$, we will repeat the previous  procedure, and by finite number of steps we will come to a vertex $v_k$  of the cycle such that
 $$\det \mathbf{L}(G)=\det \mathbf{L}(G-v_1-\cdots-v_{k-1})=\det  \mathbf{L}(C).$$
\end{proof}

Let $G = (\Gamma, \varphi)$ be a $U(\mathbb{H})$-gain graph with vertex set $\Gamma(V) = \{v_1, v_2,\ldots,v_n\}$ and edge set
 $\Gamma({E}) = \{e_1, e_2,\ldots,e_m\}$. For any $v\in \Gamma(V)$ and $e \in \Gamma({E})$, we call that $v$ and $e$ are \emph{incident} if the $(v, e)$-entry of $\mathbf{H}(G)$ is not equal to 0.
As usual, $e \in \Gamma({E})$ is exactly incident to two vertices in $\Gamma(V)$, because $e$ is considered as an edge of $\Gamma$. If $e$ is incident  only to one vertex $v$ in $\Gamma(V)$, then $e$ is called a \emph{half-edge} located at $v$. If $e$ is not incident to any vertex in $\Gamma(V)$,
 $e$ is called a \emph{free loop} of $\Gamma$.

Let $\mathbf{H}(R)$ be submatrix  of $\mathbf{H}(G)$. A \emph{reduction}
$R$ of $G$ that  corresponds to the submatrix $H(R)$ is defined to a triple $(V(R), E(R), \varphi(R))$, where $V(R)$ and $E(R)$ index the rows and columns of
$\mathbf{H}(R)$, respectively, and  $\varphi(R)$ is the restriction of $\varphi$ on $E(R)$.
 A reduction $R$ of $G$ can be considered as a graph, if $R$ does not contain  free loops but half-edges are allowed. Especially, a half-edge tree is a reduction by deleting a pendent vertex of a tree and without deleting the
edge incident to it, and preserving the gain of such an edge.

By $|S|$ we denote the cardinal of the set $S$.

\begin{lem}\label{lem:det_htr} Let $R= (V(R), E(R))$ be a half-edge tree of a $U(\mathbb{H})$-gain graph. If $|V(R)|= |E(R)|$ and  the Laplacian matrix of $R$ is $\mathbf{L}(R)$, then
\begin{align}\label{eq:det_htr}\det \mathbf{L}(R)=1.\end{align}
\end{lem}
\begin{proof}
From $|V(R)|= |E(R)|$ and  that $R= (V(R), E(R))$ contains a half-edge, it follows that  $\deg(v_j)\leq 2$ and $\exists! v_j\in V(R)$ such that $\deg(v_j)=1$, i.e.  $R$ contains a pendant vertex. Especially,  a pendant vertex and  a half-edge are on other sides on a tree. Let $|V(R)|= |E(R)|=n$. Without loss of generality, we put $v_1$ as  this unique  pendant vertex and $e_n$ by the half-edge. Then it's evident that the  incidence matrix $\mathbf{H}(G) = (\eta_{ve})$ is an upper triangular matrix with $\eta_{ii}=1$ for all $i=1,\ldots,n$ and $\eta_{i,i+1}=-\varphi(e_{i,i+1})$ for all $i = 1,\ldots,n-1$, and
$$\rdet_i \mathbf{H}(G)=\cdet_i \mathbf{H}(G)=1.$$

Similarly, for the Laplacian matrix  $\mathbf{L}(G)=(l_{ij})$, we have $l_{ii}=2$, $l_{i,i+1}=-\varphi(e_{i,i+1})$ and $l_{i+1,i}=-\overline{\varphi(e_{i,i+1}})$ for all $i = 1,\ldots,n-1$, and $l_{nn}=1$.

We put $\det \mathbf{L}(G)=\rdet_1 \mathbf{L}(G)$ and will be calculate it by Definition \ref{def:rdet}. Similarly as in the proof of Lemma \ref{lem:det_C},  we obtain the following kinds  of terms and their sums in $\rdet_1 \mathbf{L}(G)$ regarding  to  a quantity $k$ of cycles of a second order in a term of a ${\mathfrak{R}}$-determinant.
\begin{align*}k=0,&~L_1=l_{11}l_{22}\ldots l_{nn}=2^{n-1}=(-1){{n}\choose{1}}2^{n-3}=-n2^{n-3},\\
  k=2,&~L_3=(-1)^2\left[2(n-3)+{{n-3}\choose{2}}\right]2^{n-5}=\frac{n-3}{2}n2^{n-5},\\
 \end{align*}
We have
 \begin{align*}
  &~L_{k+1}=(-1)^k\frac{(n-k-1)!}{k!(n-2k)!}n2^{n-2k},~~~{\rm ~for~ any}~~~k\leq\left[\frac{n}{2}\right].
 \end{align*}
Because of
\begin{align*}
\det \mathbf{L}(G)=\rdet_1 \mathbf{L}(G)=\sum_{k=0}^m {(-1)^k\frac{(n-k-1)!}{k!(n-2k)!}n2^{n-2k-1}}=1,
\end{align*}
where $m=\left[\frac{n}{2}\right]$ is the integer part of $n$, Eq. \eqref{eq:det_htr} holds.
\end{proof}
Given a $U(\mathbb{H})$-gain graph
$G$, a maximal connected subgraph of $G$ is called a \emph{component} of $G$. Each subgraph of a gain graph is also referred as a gain graph.
If each component of a reduction  $R\subseteq G$  has an equal number of vertices and edges, then we say that the reduction  $R$ is \emph{unicycle-like}. Therefore, if $G_1$ is a component of a unicycle-like reduction of $G$, then we have two cases, $G_1$ is either
unicyclic or a half-edge tree. If  $G_1$ is a unicyclic graph with the unique cycle $C$, then from Lemma \ref{lem:det_unic}, $\det \mathbf{L}(G_1) = \det \mathbf{L}(C)$.
If $G_1$ is a half-edge tree, then by Lemma \ref{lem:det_htr}, $\det \mathbf{L}(G_1) = 1$.

As usual, for  any pair of matrices $\mathbf{A}$ of size $m \times n$ and $\mathbf{B}$ of size $p \times q$, the direct sum of $\mathbf{A}$ and $\mathbf{B}$ is a matrix of size $(m + p) \times (n + q)$ defined as
$${\mathbf{A}}\bigoplus \mathbf{B}=\begin{bmatrix}\mathbf{A}&0\\0& \mathbf{B}\end{bmatrix}.$$
\begin{lem}\label{lem:det_R}
Let $R = (V(R), E(R))$ be a reduction of a given  $U(\mathbb{H})$-gain graph with $|V(R)| = |E(R)|$. If $R$ is not a unicycle-like
reduction, then $\det \mathbf{L}(R) = 0$. Suppose that $R$ is unicycle-like, then $\det \mathbf{L}(R) = 0$ when any one of the components is a
balanced unicyclic graph; otherwise
\begin{align*}
\det \mathbf{L}(R)=\prod_S \det \mathbf{L}(S)=\prod_C \det \mathbf{L}(C)=\prod_C |1-\varphi(C)|^2,
\end{align*}
where $S$ is taken over all components of $R$, and $C$ is taken over all cycles in $R$.
\end{lem}
\begin{proof}Note that if $|V(S)|<|E(S)|$ for any component $S\subseteq R$, then similarly as in Lemma \ref{lem:det_T}, $\det \mathbf{L}(S)=0$. If $|V(S_1)|>|E(S_1)|$, where $S_1\subset R$, then $S_2\subset R$ such that $|V(S_2)|<|E(S_2)|$ and $\det \mathbf{L}(S_2)=0$. Suppose that $S_1,S_2,\ldots,S_k$ is all unicyclic components of $R$, whose are not balance, and $|V(S_i)|=|E(S_i)|$ for all $i=1,\ldots,k$. Note that some of $S_i\subseteq C$  are unicyclic and others are half-edge trees. Let $\mathbf{L}(R)=\bigoplus_{i=1}^k \mathbf{L}(S_i)$. Since $\mathbf{L}(S_i)$ is Hermitian for all $i=1,\ldots,k$, then by Lemmas \ref{lem:det_unic} and  \ref{lem:det_htr}
\begin{align*}
\det \mathbf{L}(R)=\rdet_1  \mathbf{L}(R)=\prod_{i=1}^k\rdet_1  \mathbf{L}(S_i)=\prod_{i=1}^k\det  L(S_i)=\prod_C |1-\varphi(C)|^2.
\end{align*}

\end{proof}
The next theorem is regarding the determinant  of a Laplacian matrix of an arbitrary $U(\mathbb{H})$-gain graph.
\begin{thm} \label{thm:fin} Let $G$ be a $U(\mathbb{H})$-gain graph and $\mathbf{L}(G)$ be its Laplacian matrix. Then
\begin{align}\label{eq:det_G}
\det \mathbf{L}(G)=\sum_R\prod_S \det \mathbf{L}(S)=\sum_R\prod_C \det \mathbf{L}(C)=\prod_C |1-\varphi(C)|^2,
\end{align}
where  the sum is taken over all unicycle-like reductions $R$ of $G$, $S$ is taken over all components of $R$, and $C$ is taken over all cycles
in $R$.
\end{thm}
\begin{proof}Consider an arbitrary reduction $R\subseteq G$ having $|V(G)|$ vertices of $G$, i.e. $|V(R)|=|V(G)|$ and $|V(R)|=|E(R)|$ by definition of a reduction. It is evidently that   $|E(R)|\leq|E(G)|$,  otherwise a reduction $R$ should contain free loops. If $|E(R)|=|E(G)|$, then such reduction  $R$ is unique in  $G$ and \eqref{eq:det_G} holds due to Lemma \ref{lem:det_R}.

Let   $|E(R)|<|E(G)|$, and put $|E(R)|=n$ and $|E(G)|=m$.
 Then  $\mathbf{H}(R)$ is a $(n\times n)$-submatrix of $\mathbf{H}(G)\in \mathbb{H}^{n\times m}$, $\mathbf{H}(R)^*$ is a corresponding submatrix of $\mathbf{H}(G)^*\in \mathbb{H}^{m\times n}$, and $\mathbf{L}(R)=\mathbf{H}(R)\mathbf{H}(R)^*\in \mathbb{H}^{n\times n}$. Denote the matrix $\widetilde{\mathbf{L}(R)}=\mathbf{H}(R)^*\mathbf{H}(R)\in \mathbb{H}^{n\times n}$ that  is a principal submatrix of $\widetilde{\mathbf{L}(G)}=\mathbf{H}(G)^*\mathbf{H}(G)\in \mathbb{H}^{m\times m}$ for any reduction $R$. By Lemma \ref{lem:eqdet},
 $$\det \mathbf{L}(G)= \sum_{\alpha\in I_{n,m}} \det (\widetilde{\mathbf{L}(G)})_\alpha^\alpha,
   $$
   where
   $  (\widetilde{\mathbf{L}(G)})_\alpha^\alpha $ is  a principal submatrix of $\widetilde{\mathbf{L}(G)}$
   whose rows and columns are indexed by $\alpha : = \left\{
{\alpha _{1},\ldots,\alpha _{n}} \right\} \subseteq {\left\{
{1,\ldots ,m} \right\}}$, and
 $I_{n,
m}: = {\left\{ {\alpha :\,\, 1 \le \alpha_{1} < \cdots< \alpha_{n} \le m} \right\}}.
$
For any  reduction $R$ of $G$ with  $|V(R)|=|V(G)|$, we have that $\widetilde{\mathbf{L}(R)}=  (\widetilde{\mathbf{L}(G)})_\alpha^\alpha $ for some $\alpha\in I_{n,m}$.
Since by Lemma \ref{lem:crit_nonsing} we have  $\det \widetilde{\mathbf{L}(R)}=\det \mathbf{L}(R)$, then by summing along taking all reductions $R$ of $G$, we get
$$\det \mathbf{L}(G)= \sum_{\alpha\in I_{n,m}} \det (\widetilde{\mathbf{L}(G)})_\alpha^\alpha=\sum_R\det \widetilde{\mathbf{L}(R)}=\sum_R\det \mathbf{L}(R).
   $$
 Since each such reduction $R$ is a unicycle-like reduction of $G$, then  Lemma \ref{lem:det_R} evidently gives  $\det \mathbf{L}(R)= \prod_S \det \mathbf{L}(S)$, where $S$ is taken over all components contained in $R$. Taking into account Lemma \ref{lem:det_R} again, from this it follows \eqref{eq:det_G}.
\end{proof}
\begin{cor} Let $G$ be a $U(\mathbb{H})$-gain  graph with gains in the set $\{\pm 1, \pm \mathbf{i},\pm \mathbf{j},\pm \mathbf{k}\}$. Then for its  Laplacian matrix $\mathbf{L}(G)$, we have
$$
\det  \mathbf{L}(G)=\sum_R 4^{\omega_1}\times 2^{\omega_2},
$$
where the sum is taken over all unicycle-like reductions $R\in G$, ${\omega_1}$ and ${\omega_2}$ are the numbers of real unbalanced and
imaginary unbalanced cycles contained in  $R$, respectively.
\end{cor}
\begin{proof}The proof follows immediately from Theorem \ref{thm:fin} and definitions of real and
imaginary unbalanced cycles.
\end{proof}
\begin{thm}Let $G$ be a $U(\mathbb{H})$-gain graph and $\mathbf{L}(G)$ be its Laplacian matrix. Then
$$\det \mathbf{L}(G)=0$$
if and only if  $G$ is balanced.
\end{thm}
\begin{proof}
The proposition on that  if  $G$  is balanced, then $\det \mathbf{L}(G)=0$ is proven by Theorem \ref{thm:fin}.

Let now $\det \mathbf{L}(G)=0$.  By Theorem \ref{thm:fin},
$\det \mathbf{L}(G)=\sum_R\det \mathbf{L}(R)=\sum_R \det(\mathbf{H}(R)\mathbf{H}(R)^*)=0$, where $R$ is any reduction of $G$ with $|V(R)|=|V(G)|$. Since the Hermitian matrix $\mathbf{H}(R)\mathbf{H}(R)^*$ is semi-definite and $\det \mathbf{L}(R)=\prod_C \det \mathbf{L}(C)$ for all unicyclic subgraphs in $R$, then there does not exist a cycle $C\in R$ such that $\det \mathbf{L}(C)=0$. Hence, all reductions $R$ of $G$ and $G$ as their union are balanced.
\end{proof}

\section{An illustrative example}\label{sec:examp}

%\begin{figure}[ht]
%\begin{center}
%\includegraphics[width=2in]{fig1.png}
%\epsfig{file=fig1.png,height=40mm,width=40mm,clip=}
%\caption{fig1.png graphic}
%\end{center}
%\label{fig1}
%\end{figure}

\begin{figure}[ht]
\begin{center}
\includegraphics[width=2in]{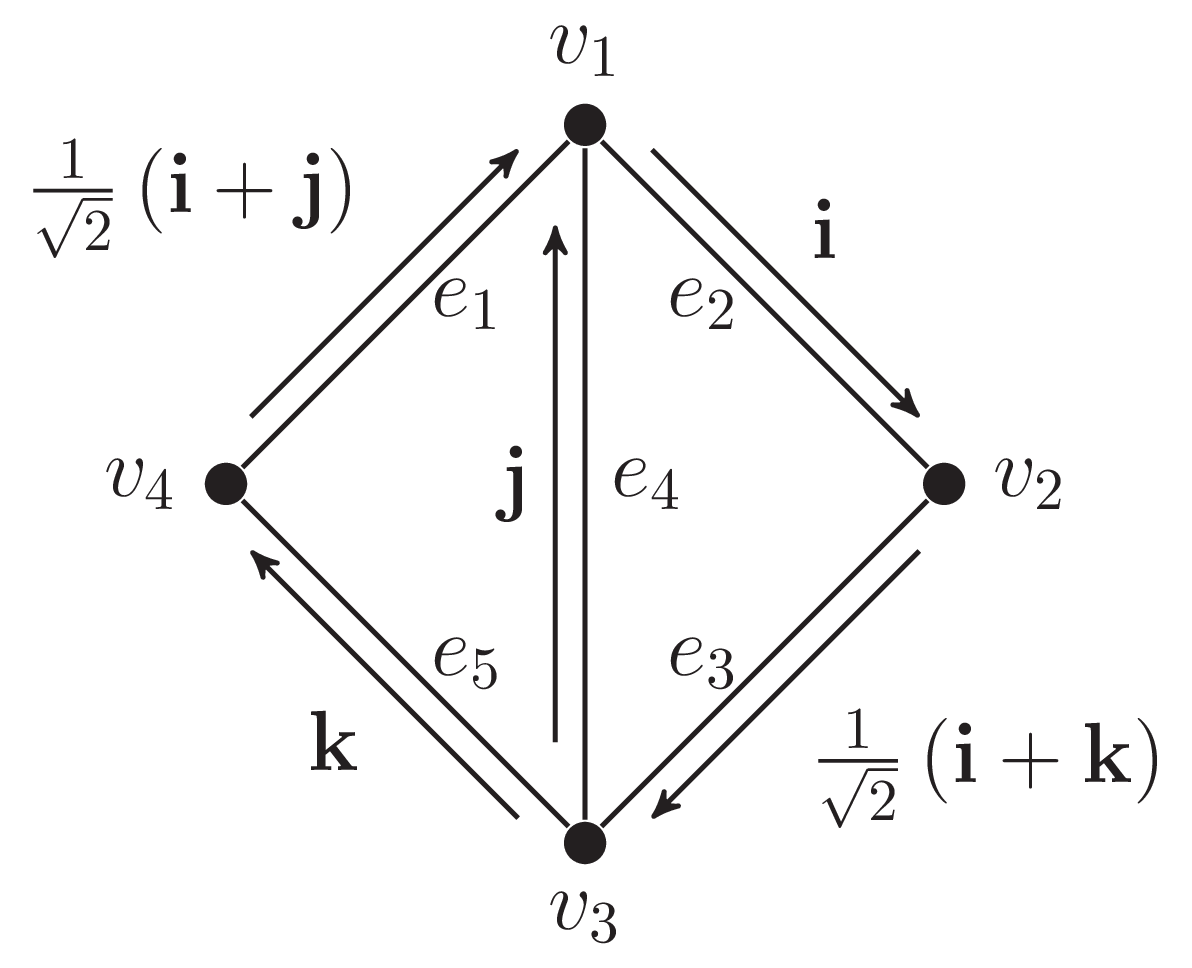}
\caption{A $U(\mathbb{H})$-gain graph G.}
\end{center}
\label{fig2}
\end{figure}
Consider a $U(\mathbb{H})$-gain graph $G$ in Fig.1. We put the gain ${\varphi(e_{41})}=\frac{1}{\sqrt{2}}(\mathbf{i}+\mathbf{j})$ on the edge $e_1$ with the direction  $\overrightarrow{v_4v_1}$. It is clear that the gain of the opposite  direction $\overrightarrow{v_1v_4}$
is ${\varphi(e_{14})}=-\frac{1}{\sqrt{2}}(\mathbf{i}+\mathbf{j})$. The same is for gains of all other oriented edges $e_i$, $i=2,\ldots,5$.
We have the following incidence matrices for the $U(\mathbb{H})$-gain graph $G$,
\begin{align*}
\mathbf{H}(G)=\begin{bmatrix}
1&-\mathbf{i}&0&1&0\\
0&1&-\frac{1}{\sqrt{2}}(\mathbf{i}+\mathbf{k})&0&0\\
0&0&1&-\mathbf{j}&-\mathbf{k}\\
-\frac{1}{\sqrt{2}}(\mathbf{i}+\mathbf{j})&0&0&0&1
\end{bmatrix}.
\end{align*}
Reductions $R$ of $G$ with $|V(R)|=|V(G)|=4$ are determined by all submatrix of forth order of the matrix $\mathbf{H}(G)$.

Especially, the submatrix $\left(\mathbf{H}(G)\right)_{\beta_1}^{\alpha}$ with the sets of row indexes  $\alpha=\{1,2,3,4\}$ and of column indexes  $\beta_1=\{1,2,3,5\}$ corresponds to the reduction $R_1$ that is the cycle $C_1= v_1e_{12}v_2e_{23} v_{3}e_{35} v_1$,  and
$$
\varphi(C_1)=\mathbf{i}\cdot\frac{1}{\sqrt{2}}(\mathbf{i}+\mathbf{k})\cdot\mathbf{k}\cdot\frac{1}{\sqrt{2}}(\mathbf{i}+\mathbf{j})=
0.5+0.5\mathbf{i}-0.5\mathbf{j}-0.5\mathbf{k}.
$$
The submatrices $\left(\mathbf{H}(G)\right)_{\beta_2}^{\alpha}$ and $\left(\mathbf{H}(G)\right)_{\beta_3}^{\alpha}$ with $\beta_2=\{1,2,3,4\}$ and $\beta_3=\{2,3,4,5\}$, respectively, correspond to the reductions $R_2$ and $R_3$ that are unicycle and both contain the cycle $C_2= v_1e_{12}v_2e_{23} v_{3}e_{31} v_1$,  and
$\varphi(C_2)=
\frac{1}{\sqrt{2}}(1-\mathbf{j}).$

Finally, the submatrices $\left(\mathbf{H}(G)\right)_{\beta_4}^{\alpha}$ and $\left(\mathbf{H}(G)\right)_{\beta_5}^{\alpha}$ with $\beta_4=\{1,2,4,5\}$ and $\beta_5=\{1,3,4,5\}$, respectively, correspond to the reductions $R_4$ and $R_5$ that are unicycle and both contain the cycle $C_3= v_1e_{13}v_3e_{34} v_{4}e_{41} v_1$,  and
$\varphi(C_3)=
\frac{1}{\sqrt{2}}(1-\mathbf{k}).$

Then by \eqref{eq:det_G}, \begin{align*}
\det \mathbf{L}(G)=\sum_R \det \mathbf{L}(C)=\sum_{k=1}^3 |1-\varphi(C_k)|^2=9-4\sqrt{2}.
\end{align*}
The same result can be obtained by direct calculation of the determinant of the Laplacian matrix.
Since
\begin{align*}
\mathbf{L}(G)=\mathbf{H}(G)\mathbf{H}(G)^*=\begin{bmatrix}
3&-\mathbf{i}&\mathbf{j}&\frac{1}{\sqrt{2}}(\mathbf{i}+\mathbf{j})\\
\mathbf{i}&2&-\frac{1}{\sqrt{2}}(\mathbf{i}+\mathbf{k})&0\\
-\mathbf{j}&\frac{1}{\sqrt{2}}(\mathbf{i}+\mathbf{k})&3&-\mathbf{k}\\
-\frac{1}{\sqrt{2}}(\mathbf{i}+\mathbf{j})&0&\mathbf{k}&2
\end{bmatrix},
\end{align*}
then for all $i=1,\ldots,4,$ $$\det \mathbf{L}(G)=\rdet_i \mathbf{L}(G)=\cdet_i \mathbf{L}(G)=9-4\sqrt{2}.$$

\section{Conclusion}\label{sec:conc}
In this paper we have extended some properties of matrix representations of a complex unit gain graph to a quaternion one. We have explored matrix representations of  a  quaternion unit gain graph such as the adjacency, Laplacian and incidence matrices. Especially, we provided a combinatorial description of the determinant of the Laplacian matrix. In carrying out this task, we inevitably encounter a problem  of defining a  determinant of a quadratic matrix with noncommutative entries (noncommutative determinant). To solve it, we use the theory of row-column determinants recently developed by one of the authors. We expect that many other results from the  theories of signed and complex unit gain graphs can be generalized to the quaternions settings by this way.

%Dijana Mosi\'c
%
%Faculty of Sciences and Mathematics, University of Ni\v s, P.O.
%Box 224, 18000 Ni\v s, Serbia
%
%{\it E-mail:} {\tt dijana@pmf.ni.ac.rs}
%
%\bigskip
%Predrag S. Stanimirovi\'c
%
%Faculty of Sciences and Mathematics, University of Ni\v s, P.O.
%Box 224, 18000 Ni\v s, Serbia
%
%{\it E-mail:} {\tt pecko@pmf.ni.ac.rs}
\end{document}